\documentclass[10pt]{amsart}

\usepackage{amsmath,comment}
\usepackage{amssymb}
\usepackage{enumerate}
\usepackage{amsthm}
\usepackage{enumitem}
\usepackage{hyperref}

\newcommand{\Fc}{\mathcal{F}}

\newcommand{\Nb}{\mathbb{N}}
\newcommand{\Rb}{\mathbb{R}}
\newcommand{\Kb}{\mathbb{K}}

\newtheorem{theorem}{Theorem}[section]
\newtheorem{proposition}[theorem]{Proposition}

\newtheorem{definition}[theorem]{Definition}

\newtheorem{example}[theorem]{Example}

\begin{document}

\title{Frequently hypercyclic $C_{0}$-semigroups indexed with complex sectors}

\author[Shengnan He]{Shengnan He}
\address{School of Humanities and Fundamental Sciences, Shenzhen Institute of Information Technology, Shenzhen 518172, China.}
\email{heshnmath@sziit.edu.cn}

\author[Zongbin Yin]{Zongbin Yin}
\address{School of Mathematics and Systems Science, Guangdong Polytechnic Normal University, Guangzhou 510665, China.}
\keywords{frequent hypercyclicity, complex sector,  strongly continuous semigroups, translation semigroup}
\email{yinzb\_math@163.com}
\begin{abstract}
In this paper, we study frequent hypercyclicity for strongly continuous semigroups of operators \(\left\{T_{t}\right\}_{t\in\Delta}\) indexed with complex sectors.  We propose a revised and more natural definition of frequent hypercyclicity  compared to the one in \cite{Chaouchi}.  Additionally, we establish a sufficient condition and a necessary condition for a $C_0$-semigroup \(\{T_{t}\}_{t \in \Delta}\) to be frequently hypercyclic. Moreover, we derive a practical and applicable criterion for translation semigroups \(\{T_{t}\}_{t \in \Delta}\) on $L^p_\rho(\Delta, \mathbb{K})$ spaces, expressed in terms of the integral of the weight function. As a result, we provide explicit examples of frequently hypercyclic translation semigroups on \(L^{p}_{\rho}(\Delta, \mathbb{K})\). Lastly, we present a necessary condition on the weight function for the translation semigroups, under which it is demonstrated that Example I (i) \cite{Chaouchi} is not frequently hypercyclic under the revised definition.

\end{abstract}
\maketitle

\section{Introduction}
During the last three decades, the study of chaotic dynamics generated by linear systems on infinite-dimensional vector spaces has drawn considerable attention. Research on linear chaos centered primarily on individual linear operators on topological vector spaces(see \cite{Bayart,Peris,Yin2018,He2018} and references therein), $C_0$-semigroups $\{T_t\}_{t \geq 0}$ (see~\cite{Kalmes2007,Conejero2017,Lizama,Taqbibt2024}) and $C_0$-semigroups $\{T_t\}_{t \in \Delta}$ indexed by complex sectors (see~\cite{Conejero2007,Conejero2009,Chaouchi,Hesf,He2020,He2021,Lorenzi,Liang,He2024}) on Banach spaces. Recall that a complex sector $\Delta$ is defined by $\Delta := \Delta(\alpha) = \{ re^{i\theta} \, | \, r \geq 0, \, |\theta| \leq \alpha \}$ with $0 < \alpha \leq \pi/2$. For the convenience of the reader, we also recall that a one-parameter family $\{T_{t}\}_{t\in S}$ ($S=\Rb^{+}\text{ or }\Delta$) of operators on a Banach space $X$ is called a \emph{strongly continuous semigroup of operators} or a \emph{$C_{0}$-semigroup}, if it satisfies $T_{0} = I$, $T_{t+s} = T_{t}T_{s}\ $ for all $s,t\in S$ and $\lim_{s\rightarrow t}T_{s}x=T_{t}x\ $ for all $x\in X$, $t\in S$.

For a long time, it was believed that chaos only existed in nonlinear systems. However, it has been shown that nearly all chaotic properties that may arise in nonlinear dynamical systems, such as point transitivity (hypercyclicity in linear case), topological transitivity,  topological weak mixing, Li-Yorke chaos, Devaney chaos, and positive entropy, can also occur in linear systems on infinite-dimensional spaces (see \cite{Bayart,Peris,Yin2018,He2018} and references therein).
Among others, \emph{hypercyclicity} (see \cite{Bayart,Peris,Ansari1995,Bes2016,Bes2019}) is one of the most fundamental and widely studied properties. Recall that a continuous operator $T$ (resp. a $C_0$-semigroup $\{T_t\}_{t \in S}$) on a separable topological vector space $X$ is \emph{hypercyclic} if there exists a vector $x_0 \in X$, called a \emph{hypercyclic vector}, such that the orbit $$\operatorname{Orb}(T, x_0) := \{ T^n x_0 \, | \, n \in \mathbb{Z}^+ \}(\text{resp. }\operatorname{Orb}(\{T_t\}_{t \in S}, x_0) := \{ T_t x_0 \, | \, t \in S \})$$ is dense in $X$, i.e. the meeting time set \[N(x,U):=\left\{t\in \mathbb{Z}^+(\text{resp. } S):T_tx\in U\right\}\] is not empty for every non-empty open subset $U$ of $X$. 

In recent years, frequent hypercyclicity \cite{Bayart2004,Bayart2006,Bonilla2007,GrosseErdmann2005,Mangino2011,Mangino2015,Agneessens2023,Bayart2015,Bonilla2018,Cardeccia2024,Ernst2021,Grivaux2021,Menet2017,Menet2022} has become a new research focus. The intriguing aspect of studying frequent hypercyclicity lies in its ability to quantify the frequency with which hypercyclic vectors enter each non-empty open set, requiring that the meeting time set has positive lower density. Recall that 
\begin{itemize}
\item The \emph{lower density} of a subset $A \subset \mathbb{N}_0$ is defined as
\[
\text{dens}(A) = \liminf_{N \to \infty} \frac{\text{card}\{0 \leq n \leq N \,:\, n \in A\}}{N + 1}.
\]
\item the \emph{lower density} of a measurable set $M \subset \mathbb{R}_+$ is defined by
\[
\text{Dens}(M) := \liminf_{N \to \infty} \frac{\mu(M \cap [0, N])}{N},
\]
where $\mu$ is the Lebesgue measure on $\mathbb{R}_+$. 
\item An operator $T$ (resp. a $C_0$-semigroup $(T_t)_{t \geq 0}$) on a Bananch space \(X\) is said to be \emph{frequently hypercyclic} if there exists $x \in X$ such that 
\[
\text{dens}\{n \in \mathbb{N}_0 \,:\, T^n x \in U\} > 0\ (\text{resp. } \text{Dens}(\{t \in \mathbb{R}_+ : T_t x \in U\}) > 0)
\] for any non-empty open set $U \subset X$.
\end{itemize}

However, when the time indices are extended to a complex sector $\Delta$, defining frequent hypercyclicity and proving the associated criteria become more challenging.
In 2020, Chaouchi et al. \cite{Chaouchi} first introduced and analyzed frequently hypercyclic $C_{0}$-semigroups $\{T_t\}_{t \in \Delta}$ indexed with complex sectors. 
Recall that 
\begin{itemize}
    \item (Definition 6 \cite{Chaouchi} by setting \(q=1\)) The lower density of a measurable set $A\subset \Delta$ is defined as:
    \[
    \underline{d}(A) := \liminf_{t \to \infty} \frac{m(A \cap \Delta_{t})}{t},
    \]
where $m$ is the Lebesgue measure and \(
\Delta_t := \{s \in \Delta : |s| \leq t\}\).
\item A $C_{0}$-semigroup $\{T_t\}_{t \in \Delta}$ on a Bananch space \(X\) is called \emph{frequently hypercyclic}, if there exists $x \in X$ such that 
\[
\underline{d}\{t\in \Delta :\ T_{t} x \in U\} > 0
\] for any non-empty open set $U \subset X$.
\end{itemize}

The above definition of frequent hypercyclicity for a $C_{0}$-semigroup $\{T_t\}_{t \in \Delta}$ is similar to that for a single operator and traditional $C_{0}$-semigroups $\{T_t\}_{t \in \mathbb{R}_+}$. However, the above definition of lower density provided in \cite{Chaouchi} appears to lack naturalness and soundness, according to which, any ray in \(\Delta\), defined as \(R_\theta := \{re^{i\theta} : r \geq 0\}\), has positive lower density, while it is easy to find that any ray in \(\Delta\) is sparse and considering them to have positive lower density is unreasonable.

Based on the above analysis and reasoning, we propose the following revised and more reasonable definition of lower density. In the sequel, we assume that $\Delta := \Delta(\alpha) = \{ re^{i\theta} \, | \, r \geq 0, \, |\theta| \leq \alpha \}$ with $0 < \alpha \leq \pi/2$. We denote that $m$ the Lebesgue measure and \(
\Delta_t := \{s \in \Delta : |s| \leq t\}\).
\begin{definition} 
The lower density of a measurable set $A\subset \Delta$ is defined as:
    \[
\underline{D}(A) := \liminf_{t \to \infty} \frac{m(A \cap \Delta_{t})}{m(\Delta_t)}=\liminf_{t \to \infty} \frac{m(A \cap \Delta_{t})}{\alpha t^{2}}.
    \]
\end{definition}
\begin{definition}\label{newdef}
A $C_0$-semigroup $\{T_t\}_{t \in \Delta}$ on a Banach space \(X\) is said to be frequently hypercyclic if there exists $x \in X$ such that $\underline{D}(\{t \in \Delta : T_t x \in U\}) > 0$ for any non-empty open set $U \subset X$. In this case, $x$ is called a frequently hypercyclic vector for $\{T_t\}_{t \in \Delta}$.
\end{definition}
Some sets that possess positive lower density under the definition in \cite{Chaouchi} no longer exhibit this property according to our revised definition. Consequently, a $C_0$-semigroup $\{T_t\}_{t \in \Delta}$ that is classified as frequently hypercyclic under the definition in \cite{Chaouchi} may no longer be frequently hypercyclic under the new definition. 

As an illustration, we consider an important example in the study of the dynamics of $C_0$-semigroups $\{T_t\}_{t \in \Delta}$, which was previously shown to be frequently hypercyclic under the definition in \cite{Chaouchi}. In Section 3, we will demonstrate that this semigroup is not frequently hypercyclic under our updated definition. 
\begin{example}\label{exam0}
Let $1\leq p<\infty$, $\Delta := \Delta(\pi/4) = \{ re^{i\theta} \, | \, r \geq 0, \, |\theta| \leq \pi/4 \}$ and \[
\rho(x + iy) :=
\begin{cases}
1 & \text{if } x + y \geq \sqrt{x - y}, \\
e^{x+y - \sqrt{x-y}} & \text{if } x + y < \sqrt{x - y}.
\end{cases}
\] Then the translation semigroup \(\{T_t\}_{t \in \Delta(\pi/4)}\) on \( X = L^p_\rho(\Delta(\pi/4),\mathbb{K}) \) is frequently hypercyclic under the definition in \cite{Chaouchi}, but not frequently hypercyclic under Definition \ref{newdef}. 

\end{example}
Now we clarify relevant notations used in the example above. Unless stated otherwise, these notations will be used consistently throughout the paper.

The translation semigroup $\{T_t\}_{t \in \Delta}$ on a $L^p$-space, $L^p_\rho(\Delta, \mathbb{K})$, $1 \leq p < \infty$ is given by:
\[
(T_{t}f)(x) := f(x + t), \quad x \in \Delta, \, t \in \Delta,
\]
and the $L^p$-space, $L^p_\rho(\Delta, \mathbb{K})$, is defined by
\[
L^p_\rho(\Delta, \mathbb{K}) := \left\{ u : \Delta \to \mathbb{K} : u \text{ measurable and } \|u\|_p = \left( \int_\Delta |u(\tau)|^p \rho(\tau) \, d\tau \right)^{1/p} < \infty \right\},
\]
where $\rho : \Delta \to \mathbb{R}_+$ is an \textit{admissible weight function} satisfying:
\begin{itemize}
    \item[(i)] $\rho(\cdot)$ is measurable;
    \item[(ii)] there exist constants $M \geq 1$ and $w \in \mathbb{R}$ such that
    \[
    \rho(t) \leq M e^{\omega |t'|} \rho(t + t'), \quad \text{for all } t, t' \in \Delta.
    \]
\end{itemize}
Condition (ii) ensures the translation semigroup $\{T_t\}_{t \in \Delta}$ is well-defined and strongly continuous.

To conclude the introduction, we now present the structure and main contributions of this paper. The paper is organized as follows. In Section~2, we establish a sufficient condition (Theorem~\ref{suf}) and a necessary condition (Theorem~\ref{nec}) for a $C_0$-semigroup \(\{T_{t}\}_{t \in \Delta}\) to be frequently hypercyclic. In Section~3, building on Theorem~\ref{suf} from Section~2, we derive a more practical criterion for translation semigroups \(\{T_{t}\}_{t \in \Delta}\) on $L^p_\rho(\Delta, \mathbb{K})$ spaces, expressed in terms of the integral of the weight function. As a result, we provide explicit examples of frequently hypercyclic translation semigroups on \(L^{p}_{\rho}(\Delta, \mathbb{K})\). Additionally, we present a necessary condition on the weight function for such semigroups, under which it is demonstrated that Example~\ref{exam0} is not frequently hypercyclic.

\section{Characterizations for frequent hypercyclicity }
In this section, we provide a sufficient condition in Theorem~\ref{suf} and a necessary condition in Theorem~\ref{nec} for a $C_0$-semigroup \(\{T_{t}\}_{t \in \Delta}\) to be frequently hypercyclic. 

In \cite{Bayart2006,Bonilla2007}, the authors established the well-known Frequent Hypercyclicity Criterion to demonstrate that an operator on a separable F-space is frequently hypercyclic. Inspired by these results, we present the following sufficient condition for a $C_0$-semigroup \(\{T_{t}\}_{t \in \Delta}\) to be frequently hypercyclic.

\begin{theorem}\label{suf}[Frequent Hypercyclicity Criterion ]
Let $\{T_{t}\}_{t \in \Delta}$ be a $C_0$-semigroup on a separable Banach space $X$. Assume there exist a dense subset $X_0 \subset X$ and mappings $S_t: X_0 \to X$ for $t \in \Delta$ such that the following conditions hold for all $x \in X_0$:
\begin{enumerate}[label=(\roman*), nosep]
    \item For any $\epsilon > 0$, there exists a real number $r_\epsilon > 0$ such that  
    \[
    \sup \left\{ \left\| \sum_{n \in J} T_{\lambda_n} S_{\mu_n} x \right\| : J \text{ is a finite set}, |\gamma_n| \geq r_\epsilon, \text{ and } |\gamma_n - \gamma_m| \geq 1, \, \forall n \neq m \right\} < \epsilon,
    \]  
    where $\gamma_n := |\lambda_n - \mu_n|$.

    \item $T_t S_t x \to x$ as $t \to \infty$.
\end{enumerate}

\noindent Then, the semigroup $\{T_{t}\}_{t \in \Delta}$ is frequently hypercyclic.
\end{theorem}
\begin{proof}
Choose a sequence $(y_j)$ from $X_0$ that is dense in $X$. From condition $(i)$, there exist $r_l \in \mathbb{N}$ for $l \geq 1$, such that for any $j \leq l$,
\begin{equation}\label{sumTS}
    \sup\left\{\left\|\sum_{n \in J} T_{\lambda_n} S_{\mu_n} y_j\right\| : J \text{ is a finite set}, |\gamma_n| \geq r_l, \text{ and } |\gamma_n - \gamma_m| \geq 1, \, \forall n \neq m\right\} < \frac{1}{l 2^l},
\end{equation}
where $\gamma_n := |\lambda_n - \mu_n|$.

In particular, setting $\lambda_n \equiv 0$, we obtain:
\begin{equation}\label{sumS}
   \sup\left\{\left\|\sum_{n \in J} S_{t_n} y_j\right\| : J \text{ is a finite set}, |t_n| \geq r_l, \text{ and } |t_n - t_m| \geq 1, \, \forall n \neq m\right\} < \frac{1}{l 2^l} 
\end{equation}
for any $j \leq l$.

Without loss of generality, we may assume (from condition $(ii)$) that:
\begin{equation}\label{con3}
  \|T_t S_t y_l - y_l\| < \frac{1}{2^l}, \quad \forall |t| \geq r_l, \, \forall l \in \mathbb{N}_0.  
\end{equation}

By Lemma 2.5 in \cite{Bonilla2007} and Lemma 9.5 in \cite{Peris}, there exist pairwise disjoint subsets $A(l, r_l) \subset \mathbb{N}$, $l \in \mathbb{N}$, of positive lower density such that for any $n \in A(l, r_l)$ and $m \in A(k, r_k)$,
\[
n \geq r_l \quad \text{and} \quad |n - m| \geq r_l + r_k \quad \text{if } n \neq m.
\]

Fix $l \geq 1$ and $n \in A(l, r_l)$. Let $\Delta=\Delta(\alpha).$ Without loss of generality, we assume that the length $R_n$ of the line segment $L_n := \{x + i y \in \Delta : x = n\}$ given by $R_n := 2n \tan \alpha$, satisfies $R_n > r_l$. We can select $\lfloor R_n / r_l \rfloor + 1$ evenly spaced points along $L_n$, given by:
\[
n + i n \tan \alpha, \, n + i(n \tan \alpha - r_l), \, \dots, \, n + i(n \tan \alpha - \lfloor R_n / r_l \rfloor r_l),
\]
denoted as $B^n := \{b_1^n, b_2^n, \dots, b_{\lfloor R_n / r_l \rfloor + 1}^n\}$.

As $l$ varies over $\mathbb{N}$ and $n$ varies over $A := \bigcup_{l \in \mathbb{N}} A(l, r_l)$, we construct sequences $B^n, n \in A$. Concatenating these sequences in the order of $n$ produces a new sequence:
\[
(b_k) := \bigcup_{n \in A} B^n.
\]

Define $B(l, r_l) := \bigcup_{n \in A(l, r_l)} B^n$ for $l = 1, 2, \dots$. It is straightforward to verify that the sets $B(l, r_l)$, $l \geq 1$, inherit the separation property of the sets $A(l, r_l)$. Specifically, for any distinct $b_k \in B(v, r_v)$ and $b_j \in B(\mu, r_\mu)$, we have:
\[
|b_k - b_j| \geq r_v.
\]

We construct a frequently hypercyclic vector \(x\) for the semigroup \(\{T_t\}_{t \in \Delta}\) by setting \(z_k = y_l\) if \(b_k \in B(l, r_l)\), and defining:
\begin{equation}\label{fcvector}
x = \sum_{k \in \mathbb{N}} S_{b_k} z_k.
\end{equation}
Next, we will prove two claims: first, that \(x\) is well-defined, and second, that \(x\) is frequently hypercyclic.

\section*{Step 1: Well-Definedness of the vector x}

We begin by showing that the series \(x = \sum_{k \in \mathbb{N}} S_{b_k} z_k\) is well-defined.

From the construction of \((b_k)\), for any \(l \in \mathbb{N}\), there exists an integer \(N_l\) such that \(|b_k| \geq r_l\) for all \(k \geq N_l\). For any finite set \(F \subset \mathbb{N}\), the partial sum can be expressed as:
\[
\sum_{k \in F} S_{b_k} z_k = \sum_{j=1}^\infty \sum_{\substack{b_k \in B(j, r_j) \\ k \in F}} S_{b_k} y_j.
\]
We further split this into two parts:
\[
\sum_{k \in F} S_{b_k} z_k = \sum_{j=1}^l \sum_{\substack{b_k \in B(j, r_j) \\ k \in F}} S_{b_k} y_j + \sum_{j=l+1}^\infty \sum_{\substack{b_k \in B(j, r_j) \\ k \in F}} S_{b_k} y_j.
\]

From the construction of \((b_k)\), we know that \(|b_k - b_{k'}| \geq r_j\) for any distinct pair \(b_k, b_{k'} \in B(j, r_j)\). By inequality \((\ref{sumS})\), for any \(j \leq l\) and finite \(F \subset \{N_l, N_l + 1, N_l + 2, \dots\}\), we have:
\[
\left\| \sum_{\substack{b_k \in B(j, r_j) \\ k \in F}} S_{b_k} y_j \right\| < \frac{1}{l 2^l}.
\]

Furthermore, since \(|b_k| \geq r_j\) for all \(b_k \in B(j, r_j)\), we also have, for any \(j \geq 1\) and any finite set \(F\),
\[
\left\| \sum_{\substack{b_k \in B(j, r_j) \\ k \in F}} S_{b_k} y_j \right\| < \frac{1}{j 2^j}.
\]

Combining these results, for any finite set \(F \subset \{N_l, N_l + 1, N_l + 2, \dots\}\), we obtain the total sum:
\[
\left\| \sum_{k \in F} S_{b_k} z_k \right\| \leq \sum_{j=1}^l \frac{1}{l 2^l} + \sum_{j=l+1}^\infty \frac{1}{j 2^j}< \frac{1}{2^l} + \frac{1}{2^l} = \frac{2}{2^l}.
\]

Since \(l\) is arbitrary, the series \((\ref{fcvector})\) converges unconditionally, proving that \(x\) is well-defined.
\section*{Step 2: the vector x is Frequently Hypercyclic}

We now show that \(x\) is frequently hypercyclic for the semigroup \(\{T_t\}_{t \in \Delta}\). To this end, fix \(l \geq 1\). For \(t \in B(l, r_l)\), we have:
\[
T_t x - y_l = \sum_{j=1}^{l-1} \sum_{b_k \in B(j, r_j)} T_t S_{b_k} y_j + \sum_{\substack{b_k \in B(l, r_l) \\ b_k \neq t}} T_t S_{b_k} y_l + \sum_{j=l+1}^\infty \sum_{b_k \in B(j, r_j)} T_t S_{b_k} y_j + T_t S_t y_l - y_l.
\]

For all \(j \neq l\), we have \(|t - b_k| \geq \max\{r_l, r_j\}\) and \(|b_k - b_{k'}| \geq r_j\) for any distinct \(b_k, b_{k'} \in B(j, r_j)\). Using inequality \((\ref{sumTS})\), it follows that:
\[
\left\| \sum_{b_k \in B(j, r_j)} T_t S_{b_k} y_j \right\| \leq \frac{1}{l 2^l}, \quad \forall j < l,
\]
and
\[
\left\| \sum_{b_k \in B(j, r_j)} T_t S_{b_k} y_j \right\| \leq \frac{1}{j 2^j}, \quad \forall j \geq l+1.
\]

For \(b_k \in B(l, r_l)\) with \(b_k \neq t\), we have \(|t - b_k| \geq r_l\) and \(|b_k - b_{k'}| \geq r_l\) for distinct \(b_k, b_{k'} \in B(l, r_l)\). Again, by \((\ref{sumTS})\),
\[
\left\| \sum_{\substack{b_k \in B(l, r_l) \\ b_k \neq t}} T_t S_{b_k} y_l \right\| \leq \frac{1}{l 2^l}.
\]

Additionally, since \(t \geq r_l\), it follows from condition \((\ref{con3})\) that:
\[
\left\| T_t S_t y_l - y_l \right\| \leq \frac{1}{2^l}.
\]

Combining these results, for all \(t \in B(l, r_l)\),
\[
\left\| T_t x - y_l \right\| \leq \frac{1}{ 2^l} + \frac{1}{ 2^l} + \frac{1}{2^l} = \frac{3}{2^l}.
\]

Let \(U \subset X\) be a non-empty open subset. Since \(X\) has no isolated points and \((y_l)_{l \in \mathbb{N}}\) is dense in \(X\), there exist some \(c > 0\) and infinitely many \(y_l\) with \(l \in I \subset \mathbb{N}\) such that:
\begin{equation}\label{v}
V(y_l, c) \subset U, \quad \forall l \in I,
\end{equation}
where \(V(y_l, c)\) denotes the open ball centered at \(y_l\) with radius \(c\).

By the local equicontinuity of \(\{T_t\}_{t \in \Delta}\), there exists a sufficiently large \(l_0 \in I\) such that:
\[
\|T_s z\| \leq \frac{c}{2}, \quad \forall \|z\| \leq \frac{3}{2^{l_0}}, \ |s| \leq 1.
\]

Since \(\{T_t\}_{t \in \Delta}\) is strongly continuous at \(y_{l_0}\), there exists some \(0 < \delta_0 < 1\) such that:
\[
\|T_s y_{l_0} - y_{l_0}\| \leq \frac{c}{2}, \quad \forall |s| \leq \delta_0.
\]

Now, for any \(t \in B(l_0, r_{l_0})\) and \(s \in \Delta_{\delta_0}\), we have:
\begin{equation}\label{c}
    \begin{aligned}
\|T_{t+s} x - y_{l_0}\| & \leq \|T_{t+s} x - T_s y_{l_0}\| + \|T_s y_{l_0} - y_{l_0}\| \\
                        & = \|T_s (T_t x - y_{l_0})\| + \|T_s y_{l_0} - y_{l_0}\| \\
                        & \leq \frac{c}{2} + \frac{c}{2} = c.
\end{aligned}
\end{equation}
From (\ref{v}) and (\ref{c}), it follows that:
\[
T_{t+s} x \in U, \quad \forall t \in B(l_0, r_{l_0}), \ \forall s \in \Delta_{\delta_0}.
\]

It remains to show that:
\[
\underline{D}(B(l_0, r_{l_0}) + \Delta_{\delta_0}) > 0.
\]
In fact, we can prove that:
\[
\underline{D}(B(l, r_l) + \Delta_\delta) > 0, \quad \forall l \in \mathbb{N}, 0 < \delta < 1.
\]

Fix \(l \in \mathbb{N}\) and \(0 < \delta < 1\). Let \(d_l := \text{dens}(A(l, r_l)) > 0\). For any \(\epsilon > 0\), the intersection \(A(l, r_l) \cap [0, N]\) contains at least \(\lfloor (d_l - \epsilon) N \rfloor\) integers for sufficiently large \(N \in \mathbb{N}\). By the construction of \(B(l, r_l)\), it consists of at least:
\[
\frac{\left(1 + \lfloor \frac{2 \lfloor (d_l - \epsilon) N \rfloor \tan \alpha}{r_l} \rfloor\right) \lfloor (d_l - \epsilon) N \rfloor}{2}
\]
points, leading to:
\[
\begin{aligned}
m(\left\{t\in B(l,r_{l})+\Delta_{\delta}:|t|\leq N+1 \right\}) & \geq \frac{\left(1+\lfloor \frac{2\lfloor (d_{l}-\epsilon)N \rfloor \tan\alpha}{r_{l}}\rfloor\right)\lfloor (d_{l}-\epsilon)N \rfloor}{2} \alpha\delta^{2}\\
& \geq \frac{((d_{l}-\epsilon)N-1)^{2} \alpha \delta^{2}\tan\alpha}{r_{l}}.
\end{aligned}
\]

Then we have that
\begin{equation}\label{liminfN}
\begin{aligned}
\liminf_{N\rightarrow \infty} \frac{m\left\{t\in B(l,r_{l})+\Delta_{\delta}:|t|\leq N\right\}}{\alpha N^{2}}& = \liminf_{N\rightarrow \infty} \frac{m\left\{t\in B(l,r_{l})+\Delta_{\delta}:|t|\leq N+1\right\}}{\alpha N^{2}}\\
       &\geq \frac{(d_{l}-\epsilon)  \delta^{2}\tan\alpha}{r_{l}}.
\end{aligned}
\end{equation}

Since $\epsilon>0$ was arbitrary, from (\ref{liminfN}) we obtain that
$$\liminf_{N\rightarrow \infty} \frac{m\left\{t\in B(l,r_{l})+\Delta_{\delta}:|t|\leq N\right\}}{\alpha N^{2}}\geq \frac{d_{l} \delta^{2}\tan\alpha}{r_{l}}.$$

 Hence, we have that
\begin{equation}\label{N1}
\underline{D}(B(l, r_l) + \Delta_\delta)=\liminf_{N\rightarrow \infty} \frac{m\left\{t\in B(l,r_{l})+\Delta_{\delta}:|t|\leq N\right\}}{\alpha N^{2}}>0.
\end{equation}

Thus, \(x\) is frequently hypercyclic.

\end{proof}
In \cite{Hesf,He2021}, the \(\Fc\)-transitivity of $C_0$-semigroups $\{T_{t}\}_{t \in \Delta}$ was characterized. Recall that a collection $\mathcal{F}$ of subsets of $\Delta$ is called a \textit{Furstenberg family} if it is hereditarily upward, that is, 
\[
A \in \mathcal{F}, \, B \supseteq A \implies B \in \mathcal{F}.
\]

A $C_0$-semigroup $\{T_t\}_{t \in \Delta}$ on $X$ is said to be \textit{$\mathcal{F}$-transitive} if for any non-empty open sets $U, V \subset X$, the meeting time set \[
N(U, V) := \{t \in \Delta : T_t U \cap V \neq \emptyset\}\in \Fc.
\] 

In the following theorem, we will show that if $\{T_{t}\}_{t \in \Delta}$ is frequently hypercyclic, then it is \(\Fc_{pld}\)-transitive, where $\Fc_{pld}$ denotes the Furstenberg family of all subsets of $\Delta$ having positive lower density, i.e. \(\Fc_{pld}:=\left\{A \subset \Delta: \underline{D}(A)>0\right\}\).
\begin{theorem}\label{nec}
    Let $\{T_{t}\}_{t \in \Delta}$ be a $C_0$-semigroup on a separable Banach space $X$. If $\{T_{t}\}_{t \in \Delta}$ is frequently hypercyclic, then it is \(\Fc_{pld}\)-transitive.
\end{theorem}
\begin{proof}
Let \(x\in X\) be a frequently hypercyclic vector of $\{T_{t}\}_{t \in \Delta}$. For any non-empty open subsets \(U,V\subset X\), we only need to show that \[\underline{D}(N(U,V))>0.\] 

Denote \(N(x,U):=\left\{t\in \Delta:T_tx\in U\right\}\) and \(N(x,V):=\left\{t\in \Delta:T_tx\in V\right\}.\) Let \(t_0\in N(x,U)\). From the definition of the lower density, it is not hard to check that \[\underline{D}(N(x,V))=\underline{D}(N(x,V)-t_0)\] where \[N(x,V)-t_0:=\left\{s\in \Delta:s=t-t_0\text{ for some }t\in N(x,V)\right\}.\]

For any \(s=t-t_0\in N(x,V)-t_0,\) it is not hard to see that \[ T_sT_{t_{0}}x=T_t x\in V.\]

Note that \(T_{t_{0}}x\in U.\) Then \[N(x,V)-t_0\subset N(U,V),\] and \[\underline{D}(N(U,V))>0.\]
\end{proof}

\section{Frequently hypercyclic translation semigroups on complex sectors}\label{sec:3}

In this section, we discuss the frequent hypercyclicity of translation semigroups on complex sectors. Specifically, we provide a sufficient condition, based on the integral of the weight function, for a translation semigroup \(\left\{T_{t}\right\}_{t \in \Delta}\) on \(L^{p}_{\rho}(\Delta,\mathbb{K})\) to be frequently hypercyclic, as stated in Theorem~\ref{suftran}. Using this criterion, we demonstrate that the translation semigroups from Example~\ref{exam1} to Example~\ref{exam3} is frequently hypercyclic. At last, we present a necessary condition on the weight function for the translation semigroups, under which we prove that Example~\ref{exam0} is not frequently hypercyclic.
\begin{theorem}\label{suftran}
Assume that \( 1 \leq p < \infty \) and \(\left\{T_{t}\right\}_{t \in \Delta}\) is a translation semigroup on \(X:=L^{p}_{\rho}(\Delta,\mathbb{K}).\) If 
    \[
    \int_{\Delta} \rho(s) \, ds < \infty
    \]
then the translation semigroup \(\left\{T_{t}\right\}_{t \in \Delta}\) satisfies the Frequent Hypercyclicity Criterion and is frequently hypercyclic on \( X \).
\end{theorem}
\begin{proof}
We show that the translation semigroup $\left\{T_{t}\right\}_{t\in \Delta}$ with     
\[
    \int_{\Delta} \rho(s) \, ds < \infty,
    \] satisfies the Frequent Hypercyclicity Criterion.  Let 
$$X_{0}:=\left\{f\in X:f \text{ is continuous with compact support}\right\},$$
which is dense in $X.$
For any $t\in \Delta$ and any $f\in X_{0},$ define $S_{t}f$ as follows
\begin{equation*}
S_{t}f(\tau):=\left\{
\begin{array}{lcl}
 f(\tau-t)&       & \text{ if }  \tau\in t+\Delta\\
0&       &\text{ otherwise}.
\end{array} \right.
\end{equation*}

Clearly, $T_{t}S_{t}f=f,\forall f\in X_{0},t\in \Delta.$ It suffices to show that condition $(i)$ of the Frequent Hypercyclicity Criterion (Theorem \ref{suf}) is satisfied. 

Fix $f\in X_{0}.$ Let $s_{f}:=\sup_{t\in {\rm supp}f}|t|$ and $M_{f}:=\max_{\tau\in \Delta} |f(\tau)|.$ For any $\lambda,\mu,\tau\in \Delta,$ we have that
\begin{equation*}
T_{\lambda}S_{\mu}f(\tau) =\left\{
\begin{array}{lcl}
 f(\tau+\lambda-\mu)&       & \text{ if }  \tau+\lambda\in \mu+\Delta\\
0&       &\text{ otherwise}.
\end{array} \right.
\end{equation*}
Then
\begin{equation}
\begin{aligned}
\Vert T_{\lambda}S_{\mu}f\Vert^{p}&=\int_{\left\{\tau\in \Delta:\tau+\lambda-\mu\in \Delta\right\}}|f(\tau+\lambda-\mu)|^{p}\rho(\tau)d\tau\\
&=\int_{\left\{\tau\in \Delta: \tau+\lambda-\mu\in {\rm supp} f\right\}}|f(\tau+\lambda-\mu)|^{p}\rho(\tau)d\tau\\
&\leq M_{f}^{p}\int_{\left\{\tau\in \Delta: \tau\in {\rm supp} f+\mu-\lambda\right\}}\rho(\tau)d\tau
\end{aligned}
\end{equation}
 Let $(\lambda_{n}),(\mu_{n})$ be sequences in $\Delta$ with that $|\gamma_{n}-\gamma_{m}|\geq 1,\forall n\neq m$ where $\gamma_{n}=|\lambda_{n}-\mu_{n}|.$ For any finite subset $J\subset \Nb$, 
\begin{equation}\label{finite}
\Vert\sum_{n\in J}T_{\lambda_n}S_{\mu_n}f\Vert\leq \sum_{n\in J}\Vert T_{\lambda_n}S_{\mu_n}f\Vert\leq\sum_{n\in J} M_{f}\left(\int_{\left\{\tau\in \Delta: \tau\in {\rm supp} f+\mu_n-\lambda_n\right\}}\rho(\tau)d\tau\right)^{1/p}.
\end{equation}

Since $|\gamma_{n}-\gamma_{m}|\geq 1,\forall n\neq m$ where $\gamma_{n}=|\lambda_{n}-\mu_{n}|$, then there exists a finite number $N\in \Nb$, such that for any finite subset $J\subset \Nb$ and any complex number $\tau \in \Delta$, 

 $\tau\in {\rm supp}f+\mu_n-\lambda_n$ for at most $N$ many numbers $n\in J$. 
 For any $\epsilon>0,$ let $r_\epsilon>0$ be such that 
\begin{equation}\label{rho}
\left(\int_{\left\{\tau\in \Delta: |\tau|\geq |t|-s_{f}\right\}}\rho(\tau)d\tau\right)^{1/p}<\frac{\epsilon}{M_{f}N},
\end{equation}
whenever $|t|\geq r_{\epsilon}.$ 

Note that $\tau+\lambda_n-\mu_n\in {\rm supp}f$ implies that $|\tau|>|\lambda_n-\mu_n|-s_{f}.$
Suppose that $|\lambda_n-\mu_n|\geq r_{\epsilon}.$ Then from (\ref{finite}) and (\ref{rho}) ,  we have that
\begin{equation}
\begin{aligned}
\Vert\sum_{n\in J}T_{\lambda_n}S_{\mu_n}f\Vert&\leq \sum_{n\in J} M_{f}\left(\int_{\left\{\tau\in \Delta: \tau\in {\rm supp} f+\mu_n-\lambda_n\right\}}\rho(\tau)d\tau\right)^{1/p}\\
&\leq M_{f}N\left(\int_{\left\{\tau\in \Delta: |\tau|\geq r_\epsilon-s_{f}\right\}}\rho(\tau)d\tau\right)^{1/p}<\epsilon.
\end{aligned}
\end{equation}

Now we have shown that for any fixed $f\in X_0$ and any fixed $\epsilon>0$, there exists some $r_{\epsilon}>0$ such that $$\sup\left\{\Vert\sum_{n\in J}T_{\lambda_{n}}S_{\mu_{n}}f\Vert:J \text{ is a finite set}, |\gamma_{n}|\geq r_{\epsilon} \text{ and } |\gamma_{n}-\gamma_{m}|\geq 1, n\neq m\right\}<\epsilon,$$
where $\gamma_{n}:=|\lambda_{n}-\mu_{n}|,$ which means that Frequent Hypercyclicity Criterion is satisfied. 
\end{proof}
From Theorem \ref{suftran}, we can easily obtain the following examples of frequently hypercyclic translation semigroups on complex sectors.
\begin{example}\label{exam1}
Let $\rho(\tau)=e^{-|\tau|^2},\forall \tau \in \Delta$, and $1\leq p<\infty.$ The translation semigroup $\left\{T_{t}\right\}_{t\in \Delta}$ on $X=L^{p}_{\rho}(\Delta, \Kb)$ is frequently hypercyclic.
\end{example}
\begin{proof}
It is clear that 
 \[
    \int_{\Delta} \rho(s) \, ds < \infty.
    \]
From Theorem \ref{suftran}, the translation semigroup $\left\{T_{t}\right\}_{t\in \Delta}$satisfies the Frequent Hypercyclicity Criterion and is frequently hypercyclic on \( X\).
\end{proof}
\begin{example}\label{exam2}
Let $\rho(\tau)=e^{-|\tau|},\forall \tau \in \Delta$, and $1\leq p<\infty.$ The translation semigroup $\left\{T_{t}\right\}_{t\in \Delta}$ on $X=L^{p}_{\rho}(\Delta, \Kb)$ is frequently hypercyclic.
\end{example}
\begin{proof}
The proof follows directly from the observation  
\[
\int_{\Delta} \rho(s) \, ds < \infty,
\]
combined with Theorem \ref{suftran}.
\end{proof}
\begin{example}\label{exam3}
Let $1\leq p<\infty$ and
\begin{equation*}
\rho(\tau):=\left\{
\begin{array}{lcl}
1&       & \text{ if }  \tau\in \Delta, |\tau|\leq 1\\
\frac{1}{|\tau|^{3}}&       &\text{ otherwise}.
\end{array} \right.
\end{equation*}
Then the translation semigroup $\left\{T_{t}\right\}_{t\in \Delta}$ on $X=L^{p}_{\rho}(\Delta, \Kb)$
is frequently hypercyclic.
\end{example}
\begin{proof}
The proof follows directly from the observation  
\[
\int_{\Delta} \rho(s) \, ds < \infty
\]
and Theorem \ref{suftran}.
\end{proof}

In Theorem \ref{nec}, we show that a frequently hypercyclic $C_0$-semigroup $\{T_{t}\}_{t \in \Delta}$ is \(\Fc_{pld}\)-transitive. Theorem 6 \cite{Hesf} characterized the \(\Fc\)-transitivity of translation semigroups $\{T_{t}\}_{t \in \Delta}$. Combined with Theorem \ref{nec} and Theorem 6 \cite{Hesf}, we can obtain the following necessary condition of frequently hypercyclic translation semigroups. 

Before presenting the following proposition, we revisit some definitions. 
Let $\mathcal{F}$ be a Furstenberg family. A subset $A \subset \Delta$ is called a \textit{thick $\mathcal{F}$-set}, if for any compact subset $K \subset \Delta$, there exists a $B \in \mathcal{F}$ such that $K+B \subset A$. The Furstenberg family consisting of all the thick $\mathcal{F}$-sets is denoted by $\mathcal{F}^{\text{thick}}$. Clearly, each thick $\mathcal{F}$-set must be a $\Fc$-set, that is, $\mathcal{F}^{\text{thick}}\subset \Fc.$ 

\begin{proposition}\label{nectrans}
        Let $\{T_t\}_{t \in \Delta}$ be the translation semigroup on $X = L^p_\rho(\Delta,\mathbb{K})$. If $\{T_{t}\}_{t \in \Delta}$ is frequently hypercyclic, then for any $\epsilon > 0$, $\{t \in \Delta : \rho(t) \leq \epsilon\}$ is a thick $\Fc_{pld}$-set in $\Delta$.
\end{proposition}
\begin{proof}
 In Theorem 6 \cite{Hesf}, a $\mathcal{F}$-transitive translation semigroup $\{T_t\}_{t \in \Delta}$ on $L^p_\rho(\Delta,\mathbb{K})$ can be characterized as follows:
\begin{itemize}
    \item for any $\epsilon > 0$, $\{t \in \Delta : \rho(t) \leq \epsilon\}$ is a thick $\mathcal{F}$-set in $\Delta$.
\end{itemize}

Then the proposition follows directly from Theorem \ref{nec}.
\end{proof}
From Proposition \ref{nectrans}, we can show that Example \ref{exam0} is not frequently hypercyclic under our definition.
\begin{proof}[Proof of Example \ref{exam0}]

First the translation semigroup in this example was shown to be frequently hypercyclic under the definition in \cite{Chaouchi} from the discussion in Example I (i) in the paper.

Next we prove that this semigroup is not frequently hypercyclic under our updated definition. Indeed for any fixed \( \epsilon <1 \),  the lower density of \( \{ t \in \Delta : \rho(t) \leq \epsilon \} \) is zero. For any $-1<k<1$, let $$\Delta_{1,k}:=\left\{t=x+yi\in \Delta:kx\leq y\leq x\right\}.$$ It is clear that there exists a $N>0$ such that $x + y \geq \sqrt{x - y}$ whenever $t=x+iy\in \Delta_{1,k}$ with $x>N.$ Hence  \[ \underline{D}(\{ t \in \Delta : \rho(t) \leq\epsilon \}\leq \underline{D}(\Delta/\Delta_{1,k})=\frac{k+1}{2}.\]
Letting $k\rightarrow -1$, we have that \(\underline{D}(\{ t \in \Delta : \rho(t)\leq\epsilon \}=0\), i.e. \( \{ t \in \Delta : \rho(t) \leq \epsilon \} \) is not a \(\Fc_{pld}\)-set in \(\Delta\). Therefore, from Proposition \ref{nectrans},  \(\{T_t\}_{t \in \Delta(\pi/4)}\) is not frequently hypercyclic.
\end{proof}

\section{Concluding Remarks}\label{sec:4}
The frequent hypercyclicity of single operators and a $C_{0}$-semigroup $\{T_{t}\}_{t\in \mathbb{R}_+}$ has been a research focus in linear chaotic dynamics during recent years. However, when the time indices are extended to complex sectors, defining frequent hypercyclicity and proving the associated criteria become significantly more challenging. In \cite{Chaouchi}, the authors first proposed a definition of frequent hypercyclicity for a $C_{0}$-semigroup $\{T_{t}\}_{t\in \Delta}$. However, as discussed in the introduction of this paper, that definition has certain deficiencies. Therefore, we present a revised definition in this paper, which is more natural and reasonable. Additionally, we establish a sufficient condition and a necessary condition for a $C_0$-semigroup \(\{T_{t}\}_{t \in \Delta}\) to be frequently hypercyclic. Moreover, we derive a more practical criterion for translation semigroups \(\{T_{t}\}_{t \in \Delta}\) on $L^p_\rho(\Delta, \mathbb{K})$ spaces, expressed in terms of the integral of the weight function. As a result, we provide explicit examples of frequently hypercyclic translation semigroups on \(L^{p}_{\rho}(\Delta, \mathbb{K})\). Lastly, we present a necessary condition on the weight function for the translation semigroups, under which it is demonstrated that Example I (i) \cite{Chaouchi} is not frequently hypercyclic under the revised definition.
\subsection*{Acknowledgements} This research was funded by the National Natural Science Foundation of China (No.  12101415, 62272313), Shenzhen Institute of Information Technology (No. SZIIT2022KJ008), Science and Technology Projects in Guangzhou (No. 2024A04J4429) and the project of promoting research capabilities for key constructed disciplines in Guangdong Province (No. 2021ZDJS028).

\end{document}